\documentclass{birkjour}
 \newtheorem{thm}{Theorem}[section]
 \newtheorem{cor}[thm]{Corollary}
 \newtheorem{lem}[thm]{Lemma}
 
 \theoremstyle{definition}
 
 \theoremstyle{remark}
 
 \newtheorem{ex}[thm]{Example}
 \numberwithin{equation}{section}

\begin{document}

%
%
%
%
%
%
%
%
%

\author[Youssef Aserrar and Elhoucien Elqorachi]{Youssef Aserrar and Elhoucien Elqorachi}

\address{%
	Ibn Zohr University, Faculty of sciences, 
Department of mathematics,\\
Agadir,
Morocco}

\email{youssefaserrar05@gmail.com, elqorachi@hotmail.com }
\subjclass{39B52, 39B32}

\keywords{Semigroup, cosine addition law, involutive automorphism, multiplicative function.}

\date{January 1, 2004}
\title[A generalization of the cosine addition law on  semigroups]{A generalization of the cosine addition law on semigroups}
 
\begin{abstract}
Our main result is that we describe the solutions $g,f:S\rightarrow\mathbb{C}$ of the functional equation 
\[g(x\sigma(y))=g(x)g(y)-f(x)f(y)+\alpha f(x\sigma(y)),\quad x,y\in S,\]
where $S$ is a semigroup, $\alpha \in \mathbb{C}$ is a fixed constant and $\sigma :S\rightarrow S$ an involutive automorphism.
\end{abstract}

\maketitle

\section{Introduction}
The cosine addition formula, cosine subtraction formula, and sine addition formula on any semigroup $S$   for unknown functions $g,f:S\rightarrow\mathbb{C}$ are, respectively, the functional equations
\begin{equation}
g(x\sigma(y))=g(x)g(y)-f(x)f(y),
\label{cosa}
\end{equation}
\begin{equation}
g(x\sigma(y))=g(x)g(y)+f(x)f(y),
\label{AE}
\end{equation}
\begin{equation}
g(x\sigma(y))=g(x)f(y)+f(x)g(y),
\label{sina}
\end{equation}
for all $x,y\in S$, where $\sigma :S\rightarrow S$ is an involutive automorphism. That is $\sigma (xy)=\sigma(x)\sigma(y)$ and $\sigma(\sigma(x))=x$ for all $x,y\in S$. These functional equations have been investigated by many authors. In \cite{V} Vincze obtained the solutions of \eqref{AE} on abelian groups with $\sigma =id$, and it was solved on general groups by Chung, Kannappan, and Ng \cite{Ch}. The results were extended to the case of topological groups by Poulsen and Stetk\ae r \cite{Pou} and to semigroups generated by their squares by Ajebbar and Elqorachi \cite{Ajb} with $\sigma$ an involutive automorphism. Also Stetk\ae r \cite[Theorem 6.1]{ST2} gives a description of the solution of \eqref{cosa} with $\sigma =id$ on a general semigroup in terms of the solutions of \eqref{sina} with $\sigma=id$. The most recent result about \eqref{sina}  with $\sigma=id$ was obtained by Ebanks \cite[Theorem 2.1]{EB1} and \cite[Theorem 3.1]{EB2} on semigroups. Ebanks \cite[Theorem 4.1]{EB1} gives the solution of \eqref{AE} on monoids.  Recently the authors \cite{Ase} solved \eqref{cosa}, \eqref{AE} and \eqref{sina} on semigroups.\par 
Stetk\ae r \cite[Theorem 3.1]{ST2} solved the functional equation 
\begin{equation}
g(xy)=g(x)g(y)-f(x)f(y)+\alpha f(xy),\quad x,y\in S.
\label{stet}
\end{equation}
on a semigroup $S$, where $\alpha \in \mathbb{C}$ is a fixed constant. He expressed the solutions in terms of multiplitive functions on $S$ and solutions $h:S\rightarrow\mathbb{C}$ of the special case of the sine addition law 
\begin{equation}
h(xy)=h(x)\chi(y)+h(y)\chi(x),\quad x,y\in S,
\label{h}
\end{equation}
in which $\chi :S\rightarrow\mathbb{C}$ is a multiplicative function.
Equation \eqref{stet} generalizes both the cosine addition formula \eqref{cosa} and the cosine subtraction formula \eqref{AE} with $\sigma=id$.\par 
As a continuation and a generalization of these investigations we determine the complex valued solutions of the functional equation
\begin{equation}
g(x\sigma(y))=g(x)g(y)-f(x)f(y)+\alpha f(x\sigma(y)),\quad x,y\in S,
\label{Hen}
\end{equation}
on semigroups, where $\sigma :S\rightarrow S$ is an involutive automorphism. We obtain explicit formulas for the solutions expressed in terms of multiplicative, additive and sometimes arbitrary functions. The paper concludes with some examples.\par 
Our notation is decribed in the following set up.

\section{Set up and notation}
Throughout this paper $S$ denotes a semigroup.\\ 
A function $a: S\rightarrow \mathbb{C}$ is additive if 
$$a(xy) = a(x) + a(y) \quad \text{for all}\  x, y\in S.$$ 
A function $\chi: S\rightarrow \mathbb{C}$ is multiplicative if 
$$\chi(xy) = \chi(x)\chi(y)\quad\text{for all}\  x, y \in S.$$ 
A function $f: S\rightarrow \mathbb{C}$ is central if $f(xy) = f(yx)$ for all $x, y\in S$, and $f$ is abelian if $f$ is central and $f(xyz)=f(xzy)$ for all $x,y,z\in S$.\\
For any subset $T\subseteq S$ we define the set  $T^2:=\lbrace xy\ \vert \ x, y\in T\rbrace$, so $T^2$ consists of all products of two (or more) elements of $T$.\\ 
If $\chi: S \rightarrow \mathbb{C}$ is a non-zero multiplicative function, we define the sets $$I_{\chi}:=\{x \in S \mid \chi(x)=0\}$$ 
$$P_\chi : =\{p\in I_{\chi}\backslash I_{\chi}^2\  \vert up, pv, upv\in I_{\chi}\backslash I_{\chi}^2\  \text{for all}\  u,v\in S\backslash I_{\chi}\}.$$ 
 For any function $f: S \rightarrow \mathbb{C}$ we define the function
  $$f^{*}(x)=f(\sigma(x)), x \in S.$$
  We call $f^{e}:=\frac{f+f^{*}}{2}$ the even part of $f$ and $f^{\circ}:=\frac{f-f^{*}}{2}$ its odd part. The function $f$ is said to be even if $f=f^{*}$, and $f$ is said to be odd if $f=-f^{*}$.
  In the following lemma we give some properties of the set $P_\chi$ .
\begin{lem}
\begin{enumerate}
\item[(a)] If $u\in S\backslash I_\chi$ and $p\in P_\chi$, then $up,pu\in P_\chi$ .
\item[(b)] $\sigma(P_\chi)=P_{\chi^*}$. Note in particular that $\sigma(P_\chi)=P_\chi$ , if $\chi=\chi^*$.
\end{enumerate}
\end{lem}
\begin{proof}
(a) Follows directly from the definition of $P_\chi$.\\
(b) Easy to verify using the fact that $\sigma :S\rightarrow S$ is a bijection (See \cite[Lemma 4.1]{ES1}).
\end{proof}
  For a topological semigroup $S$ let $C(S)$ denote the algebra of continuous functions from $S$ into $\mathbb{C}$.
\section{ The main result}
The following lemmas will be used later. 
\begin{lem}
Let $f,g:S\rightarrow\mathbb{C}$ be a solution of the functional equation 
\begin{equation}
f(x\sigma(y))=\beta f(x)f(y)-\beta g(x)g(y),
\label{FF3}
\end{equation}
where $g$ a non-zero function such that $g=0$ on $S^2$ and $\beta \in \mathbb{C}\backslash \lbrace 0\rbrace$ is a constant. Then $f$ and $g$ are linearly dependent.
\label{lemme}
\end{lem}
\begin{proof}
If $f=0$ or $g=0$ then $f$ and $g$ are linearly dependent, so we may assume that $f\neq 0$ and $g\neq 0$. By applying \eqref{FF3} to the pair $(xy,z)$ and taking into account that $g=0$ on $S^2$ we obtain 
\begin{equation}
f(xy\sigma(z))=\beta f(xy)f(z),\ \text{for all}\ x,y,z\in S.
\label{FF4}
\end{equation}
Now if we apply \eqref{FF3} to the pair $(x,\sigma(y)z)$ and taking into account that $g=0$ on $S^2$ we get
\begin{equation}
f(xy\sigma(z))=\beta f(x)f(\sigma(y)z),\ \text{for all}\ x,y,z\in S.
\label{FF5}
\end{equation}
So we deduce from \eqref{FF4} and \eqref{FF5} since $\beta \neq 0$ that 
\[f(xy)f(z)=f(x)f(\sigma(y)z),\ \text{for all}\ x,y,z\in S.\]
Then we get since $f\neq 0$ that $f(xy)=f(x)\phi(\sigma(y))$ for some function $\phi :S\rightarrow\mathbb{C}$. Substituting this into \eqref{FF3} we find that
\[f(x)\phi(y)=\beta f(x)f(y)-\beta g(x)g(y).\]
This implies that \[\left[\beta f(y)-\phi (y) \right]f=\beta g(y)g. \]
Choosing $y_0\in S$ such that $g(y_0)\neq 0$ we see that $f$ and $g$ are linearly dependent. This completes the proof of Lemma \ref{lemme}.
\end{proof}

\begin{lem}
Let $g,f :S\rightarrow\mathbb{C}$ be  two functions such that 
\[f=a_1\chi_1+a_2\chi_2\quad\text{and}\quad g=b_1\chi_1+b_2\chi_2,\]
where $\chi_1,\chi_2:S\rightarrow\mathbb{C}$ are two different non-zero multiplicative functions, $g\neq 0$ and $a_1,a_2\in \mathbb{C}\backslash \lbrace 0\rbrace$, $b_1,b_2\in \mathbb{C}$ are constants. We have 
\begin{enumerate}
\item[(1)] If $f$ is even and $g$ is odd, then $a_1=a_2$ and $b_1+b_2=0$.
\item[(2)] If $f$ is odd and $g$ is even, then $a_1+a_2=0$ and $b_1=b_2$.
\label{prob}
\end{enumerate}
\end{lem}
\begin{proof}

(1) If $f$ is even and $g$ is odd, we get that 
\begin{equation}
a_1\chi_1+a_2\chi_2=a_1\chi_1^*+a_2\chi_2^*,
\label{eq1}
\end{equation}
and 
\begin{equation}
b_1\chi_1+b_2\chi_2=-b_1\chi_1^*-b_2\chi_2^*.
\label{eq2}
\end{equation}
By the help of \cite[Proposition A.2]{ST2} we deduce from \eqref{eq1} that $a_1\chi_1=a_1\chi_1^*$ or $a_1\chi_1=a_2\chi_2^*$. If $a_1\chi_1=a_1\chi_1^*$ then $\chi_1=\chi_1^*$ since $a_1\neq 0$, and then we get from \eqref{eq1} that $\chi_2=\chi_2^*$ since $a_2\neq 0$. In view of \eqref{eq2} we deduce that $g=0$. This contradicts the fact that $g\neq 0$. So $a_1\chi_1=a_2\chi_2^*$ then since $\chi_1,\chi_2$ are non-zero, $a_1\neq 0$ and $a_2\neq 0$, according to \cite[Theorem 3.18 (a)]{ST1} we deduce that $\chi_1=\chi_2^*$ and then  $a_1=a_2$. Now \eqref{eq2} becomes $(b_1+b_2)(\chi_1+\chi_2)=0$ which implies that $b_1+b_2=0$. This is case (1).\\
(2) We proceed similarly to the case (1) to get the result. This completes the proof of Lemma \ref{prob}.
\end{proof}
In the following lemma we give some key properties of the solutions of Equation \eqref{Hen}.
\begin{lem}
Let $g,f :S\rightarrow\mathbb{C}$ be a solution of the functional equation \eqref{Hen}, and define the function $G:=g-\alpha f$. The following statements hold:
\begin{enumerate}
\item[(1)] $G(x\sigma(y))=G(y\sigma(x))$ for all $x,y\in S$.
\item[(2)] $G(xyz)=G^*(xyz)$ for all $x,y,z\in S$.
\item[(3)] For all $x,y,z\in S$
\begin{equation}
g^e(x)g^{\circ}(yz)=f^e(x)f^{\circ}(yz),
\label{L1}
\end{equation}
\begin{equation}
g^e(yz)g^{\circ}(x)=f^e(yz)f^{\circ}(x).
\label{L2}
\end{equation}
\end{enumerate}
\label{Lem1}
\end{lem}
\begin{proof}
(1) The functional equation \eqref{Hen} is equivalent to 
\begin{equation}
G(x\sigma(y))=g(x)g(y)-f(x)f(y)\quad \text{for all}\quad x,y\in S.
\label{p1}
\end{equation}
The right hand side of \eqref{p1} is invariant
under the interchange of $x$ and $y$. That is 
\begin{equation}
G(x\sigma(y))=G(y\sigma(x))\quad \text{for all}\quad x,y\in S.
\label{rat1}
\end{equation}
(2) Replacing $y$ by $\sigma(y)$ in \eqref{rat1}, we get $G(xy)=G^*(yx)$ for all $x,y\in S$, and then  we deduce that 
\[G(xyz)=G^*(zxy)=G(yzx)=G^*(xyz)\quad \text{for all}\quad x,y,z\in S.\]
(3) By applying the identity \eqref{p1} to the pair $(x,\sigma(yz))$ we obtain
\begin{equation}
G(xyz)=g(x)g^*(yz)-f(x)f^*(yz)\quad \text{for all}\quad x,y,z\in S.
\label{p2}
\end{equation}
Then by using (2), we deduce from \eqref{p2} that 
\begin{equation}
g(x)g^*(yz)-f(x)f^*(yz)=g^*(x)g(yz)-f^*(x)f(yz)\quad \text{for all}\quad x,y,z\in S.
\label{p3}
\end{equation}
Since $k=k^e+k^{\circ}$ and $k^*=k^e-k^{\circ}$ for any function $k:S\rightarrow\mathbb{C}$, then we get from \eqref{p3} after some rearrangement that for all $x,y,z\in S$
\begin{equation}
g^{\circ}(x)g^e(yz)-f^{\circ}(x)f^e(yz)=g^e(x)g^{\circ}(yz)-f^e(x)f^{\circ}(yz).
\label{p4}
\end{equation}
In the identity \eqref{p4} the left hand side is an odd function of $x$ while the right is an even function of $x$, then 
\begin{equation}
g^{\circ}(x)g^e(yz)-f^{\circ}(x)f^e(yz)=0,
\end{equation}
and 
\begin{equation}
g^e(x)g^{\circ}(yz)-f^e(x)f^{\circ}(yz)=0,
\end{equation}
 for all $x,y,z\in S$. This completes the proof of Lemma \ref{Lem1}.
\end{proof}
Now we present the general solution of \eqref{Hen} on semigroups. Stetk\ae r \cite[Theorem 3.1]{ST2} is Theorem \ref{thm1} with $\sigma =id$.
\begin{thm}
The solutions $g,f:S\rightarrow\mathbb{C}$ of the functional equation \eqref{Hen} are the following families:
\begin{enumerate}
\item[(1)] $\alpha = \pm 1$, $f$ is any non-zero function and $g=\alpha f$.
\item[(2)]$\alpha \neq 1$, $f=g\neq 0$ and $g=0$ on $S^2$.
\item[(3)] $\alpha \neq -1$, $f=-g\neq 0$ and $g=0$ on $S^2$.
\item[(4)] $f=(q+\alpha)\dfrac{\chi}{2}$ and $g=\left(1\pm \sqrt{1+q^2-\alpha ^2} \right)\dfrac{\chi}{2} $, where $q\in \mathbb{C}$ is a constant and $\chi :S\rightarrow \mathbb{C}$ a non-zero even multiplicative function.
\item[(5)] $f=\alpha \dfrac{\chi_1+\chi_2}{2}+q\dfrac{\chi_1-\chi_2}{2}$ and $g=\dfrac{\chi_1+\chi_2}{2}\pm \sqrt{1+q^2-\alpha^2}\dfrac{\chi_1-\chi_2}{2}$, where $\chi_1,\chi_2:S\rightarrow\mathbb{C}$ are two different non-zero even multiplicative functions and $q\in \mathbb{C}\backslash \lbrace \pm \alpha\rbrace$ is a constant.
\item[(6)] $\alpha \neq 0$,  $f=\alpha \chi_1$ and $g=\chi_2$, where $\chi_1,\chi_2:S\rightarrow\mathbb{C}$ are two different non-zero even multiplicative functions.
\item[(7)] $f=\alpha \chi+h$ and $g=\chi\pm h$, where $\chi:S\rightarrow\mathbb{C}$ is a non-zero even multiplicative function and $h:S\rightarrow\mathbb{C}$ is an even solution of the special sine addition law \eqref{h}.
\item[(8)] $\alpha\neq \pm 1$, $f=\dfrac{1+\alpha}{2}\chi-\dfrac{1-\alpha}{2}\chi^*$ and $g=\dfrac{1+\alpha}{2}\chi+\dfrac{1-\alpha}{2}\chi^*$, where $\chi:S\rightarrow\mathbb{C}$ is a multiplicative function such that $\chi\neq \chi^*$.\\
Note that, off the exceptional case (1), $g$ and $f$ are abelian.\par 
Moreover, if $S$ is a topological semigroup,  $f\in C(S)$ such that $f\neq \alpha \chi$ for any multiplicative function $\chi\in C(S)$, then $g\in C(S)$.
\end{enumerate}
\label{thm1}
\end{thm}
\begin{proof}
It is easy to check that each of the pairs described in Theorem \ref{thm1} is a solution of \eqref{Hen}. So assume that the pair $(g,f)$ is a solution of equation \eqref{Hen}. If $g=0$, then \eqref{Hen} can be written as
\[\alpha f(x\sigma(y))=f(x)f(y),\ x,y\in S.\]
If $\alpha =0$ then $f=0$. This occurs in case (4) with $q=-\alpha$. Now if $\alpha \neq 0$, then
\[f(x\sigma(y))=\dfrac{1}{\alpha}f(x)f(y),\ x,y\in S.\]
Then according to \cite[Lemma 4.1]{Ase}, $f=\alpha \chi$, where $\chi :S\rightarrow\mathbb{C}$ is an even multiplicative function. This occurs in case (4) with $q=\alpha$. From now on we assume that $g\neq 0$ and we discuss two cases according to wether $g$ and $f$ are linearly dependent or not.\\
\underline{First case:} $g$ and $f$ are linearly dependent. There exists a constant $c\in \mathbb{C}$ such that $f=cg$. So \eqref{Hen} becomes
\[(1-\alpha c)g(x\sigma(y))=(1-c^2)g(x)g(y),\ \text{for all}\ x,y\in S.\] 
Then we obtain by proceeding exactly as in the proof of \cite[Lemma 4.3]{ST2} and using \cite[Lemma 4.1]{Ase} the cases (1), (2), (3) and (4).\\
\underline{Second case :} $g$ and $f$ are linearly independent.\\
\underline{Subcase A :} $g^{\circ}=0$ and $f^{\circ}=0$. That is $g$ and $f$ are even, so if we replace $y$ by $\sigma(y)$, the functional equation \eqref{Hen} can be written as follows
\[g(xy)=g(x)g(y)-f(x)f(y)+\alpha f(xy)\quad x,y\in S.\]
According to \cite[Theorem 3.1]{ST2} and taking into account that $f$ and $g$ are linearly independent and $g\neq 0$ we have the following possibilities :\\
(i) $f=\alpha \dfrac{\chi_1+\chi_2}{2}+q\dfrac{\chi_1-\chi_2}{2}$ and $g=\dfrac{\chi_1+\chi_2}{2}\pm \sqrt{1+q^2-\alpha^2}\dfrac{\chi_1-\chi_2}{2}$, where $\chi_1,\chi_2:S\rightarrow\mathbb{C}$ are two different non-zero multiplicative functions and $q\in \mathbb{C}\backslash \lbrace \pm \alpha\rbrace$ is a constant. So 
\[f-\alpha g=\dfrac{q\pm \alpha \sqrt{1+q^2-\alpha^2}}{2}\left(\chi_1-\chi_2 \right). \]
Since $f$ and $g$ are even, we see that
\begin{equation}
\dfrac{q\pm \alpha \sqrt{1+q^2-\alpha^2}}{2}\left(\chi_1-\chi_2 \right)=\dfrac{q\pm \alpha \sqrt{1+q^2-\alpha^2}}{2}\left(\chi_1^*-\chi_2^* \right).
\label{par1}
\end{equation} 
On the other hand $g$ and $f$ are linearly independent implies that $f-\alpha g\neq 0$, and then $\dfrac{q\pm \alpha \sqrt{1+q^2-\alpha^2}}{2}\neq 0$. So \eqref{par1} reduces to
\[\chi_1-\chi_2=\chi_1^*-\chi_2^*.\]
Then since $g$ is even, we get that 
\[\chi_1+\chi_2=\chi_1^*+\chi_2^*.\]
So, we deduce from the last two identities that $\chi_1=\chi_1^*$ and $\chi_2=\chi_2^*$. This occurs in case (5).\\
(ii) $\alpha \neq 0$,  $f=\alpha \chi_1$ and $g=\chi_2$, where $\chi_1,\chi_2:S\rightarrow\mathbb{C}$ are two different non-zero multiplicative functions. In addition since $f$ and $g$ are even, we deduce that $\chi_1^*=\chi_1$ and $\chi_2^*=\chi_2$. This is case (6).\\
(iii) $f=\alpha \chi+h$ and $g=\chi\pm h$, where $h,\chi:S\rightarrow\mathbb{C}$ is a solution of the sine addition law 
\[h(xy)=h(x)\chi(y)+h(y)\chi(x),\quad x,y\in S,\]
such that $\chi\neq 0$ is multiplicative and $h\neq 0$. So $f-\alpha g=(1\pm\alpha)h$, and $1\pm\alpha\neq 0$ since $f$ and $g$ are linearly independent. Since $f$ and $g$ are even, we get that $h^*=h$ and then  $\chi^*=\chi$. This occurs in case (7).\\
\underline{Subcase B:} $g^{\circ}=0$ and $f^{\circ}\neq 0$. In this case we deduce from \eqref{L1} and \eqref{L2} respectively that $f^e(x)f^{\circ}(yz)=0$ and $f^e(yz)=0$ for all $x,y,z\in S$.\\
\underline{Subcase B.1:} $f^{\circ}=0$ on $S^2$. That is $f=0$ on $S^2$ since $f^e=0$ on $S^2$. The functional equation \eqref{Hen} can be written as follows 
\[g(x\sigma(y))=g(x)g(y)-f(x)f(y),\ \text{for all}\  x,y\in S.\]
Since $f\neq 0$ ($f^{\circ}\neq 0$) we get according to Lemma \ref{lemme}  that $f$ and $g$ are linearly dependent. This is a contradiction. This case does not occur.\\
\underline{Subcase B.2:} $f^{\circ}\neq 0$ on $S^2$. This implies that $f^e=0$, so if we replace $y$ by $\sigma(y)$, the functional equation \eqref{Hen} can be written as follows :
\[g(xy)=g(x)g(y)+f(x)f(y)+\alpha f(xy),\quad x,y\in S.\]
This means that the pair $(g,if)$ satisfies \eqref{stet}, where $\alpha$ is replaced by $-i\alpha$. According to \cite[Theorem 3.1]{ST2} and taking into account that $g$ and $f$ are linearly independent and $g\neq 0$ we have the following possibilities:\\
(i)  $if=-i\alpha \dfrac{\chi_1+\chi_2}{2}+q\dfrac{\chi_1-\chi_2}{2}$ and $g=\dfrac{\chi_1+\chi_2}{2}\pm \sqrt{1+q^2+\alpha^2}\dfrac{\chi_1-\chi_2}{2}$, where $\chi_1,\chi_2:S\rightarrow\mathbb{C}$ are two different non-zero multiplicative functions and $q\in \mathbb{C}\backslash \lbrace \pm i\alpha\rbrace$ is a constant. So
\[f=\left(\dfrac{-\alpha-iq}{2} \right)\chi_1+\left(\dfrac{iq-\alpha}{2} \right)\chi_2,\]
and 
\[g=\left(\dfrac{1+\sqrt{1+q^2+\alpha^2}}{2} \right)\chi_1+\left(\dfrac{1-\sqrt{1+q^2+\alpha^2}}{2} \right)\chi_2,\]
or 
\[g=\left(\dfrac{1-\sqrt{1+q^2+\alpha^2}}{2} \right)\chi_1+\left(\dfrac{1+\sqrt{1+q^2+\alpha^2}}{2} \right)\chi_2.\]
Since $q\neq \pm i\alpha$, we see that $\alpha \neq \pm iq$. Then since $f$ is odd and $g$ is even we get  according to Lemma \ref{prob} (2) that 
\[\dfrac{-\alpha-iq}{2} +\dfrac{iq-\alpha}{2}=0\]
and 
\[\dfrac{1-\sqrt{1+q^2+\alpha^2}}{2}=\dfrac{1+\sqrt{1+q^2+\alpha^2}}{2}.\]
Then $\alpha=0$ and $\sqrt{1+q^2+\alpha^2}=0$. This implies that $iq=\pm 1$, so 
\[f=\pm\left(\dfrac{\chi_2-\chi_1}{2} \right),\]
and 
\[g=\pm \left(\dfrac{\chi_2-\chi_1}{2}\right).\]
 This contradicts the fact that $g$ and $f$ are linearly independent. This case does not occur.\\
(ii) $\alpha \neq 0$,  $if=-i\alpha \chi_1$ and $g=\chi_2$, where $\chi_1,\chi_2:S\rightarrow\mathbb{C}$ are two different non-zero multiplicative functions. Then $f=-\alpha \chi_1$. In addition since $f$ is odd and $g$ is even, we deduce that $\chi_1^*=-\chi_1$ and $\chi_2^*=\chi_2$. This implies that $\chi_1=0$. This case does not occur.\\
(iii) $if=-i\alpha \chi+h$ and $g=\chi\pm h$, where $\chi,h:S\rightarrow\mathbb{C}$ is a solution of \eqref{h}. Then $f=-\alpha \chi-ih$. If $g=\chi+h$, we get that 
$$f+ig=(i-\alpha)\chi,$$ 
and then since $f$ is odd and $g$ is even 
$$ig-f=(i-\alpha)\chi^*.$$
By adding the last two identities we obtain $g=(i-\alpha )\dfrac{\chi+\chi^*}{2i}$, and then $f=(i-\alpha )\dfrac{\chi-\chi^*}{2}$. The fact the $f\neq 0$ implies that $\chi\neq \chi^*$ and $\alpha \neq i$. Substituting the forms of $f$ and $g$ in \eqref{Hen} we get after reduction that
\[\chi(x)\left[(1-i\alpha )\chi^*(y)-(1+i\alpha)\chi(y) \right]+(1+i\alpha)\chi^*(x)\left[\chi(y)-\chi^*(y) \right] =0, \]
for all $x,y\in S$. Then since $\chi\neq \chi^*$ we deduce that $1+i\alpha=0$. That is $\alpha =i$, so this case does not occur since $\alpha \neq i$. For the case $g=\chi-h$ we proceed by the same way to get that $g=(i+\alpha )\dfrac{\chi+\chi^*}{2i}$ and  $f=(i+\alpha )\dfrac{\chi-\chi^*}{2}$ with $\chi\neq \chi^*$ and $\alpha \neq -i$. Similarly to the previous case we get by substitution that $\alpha =-i$. This case does not occur.
\\
\underline{Subcase C :} $g^{\circ} \neq 0$ and $f^{\circ} =0$. We deduce from \eqref{L1} and \eqref{L2} respectively that $g^e(x)g^{\circ}(yz)=0$ and $g^e(yz)=0$ for all $x,y,z\in S$.\\
\underline{Subcase C.1 :} $g^{\circ} = 0$ on $S^2$. Then $g=0$ on $S^2$ since $g^e=0$ on $S^2$, so the functional equation \eqref{Hen} becomes 
\[0=g(x)g(y)-f(x)f(y)+\alpha f(x\sigma(y)),\quad x,y\in S.\]
That is 
\[\alpha f(x\sigma(y))=f(x)f(y)-g(x)g(y) ,\quad x,y\in S.\]
Since $f$ and $g$ are linearly independent, we see that $\alpha \neq 0$, so for all $x,y\in S$

\[f(x\sigma(y))=\dfrac{1}{\alpha} f(x)f(y)-\dfrac{1}{\alpha} g(x)g(y).\]
Since $g\neq 0$ ($g^{\circ}\neq 0$)  we get according to Lemma \ref{lemme}  that $f$ and $g$ are linearly dependent. This is a contradiction. This case does not occur.\\
\underline{Subcase C.2 :} $g^{\circ} \neq 0$ on $S^2$. In this case we get that $g^e=0$, so if we replace $y$ by $\sigma(y)$ the functional equation \eqref{Hen} can be written as follows 
\[g(xy)=-g(x)g(y)-f(x)f(y)+\alpha f(xy),\quad x,y\in S.\]
This means that the pair $(-g,if)$ satisfies the functional equation \eqref{stet}, where $\alpha$ is replaced by $i\alpha$. According to \cite[Theorem 3.1]{ST2} and taking into account that $g$ and $f$ are linearly independent we have the following cases\\
(i) $if=i\alpha \dfrac{\chi_1+\chi_2}{2}+q\dfrac{\chi_1-\chi_2}{2}$ and $-g=\dfrac{\chi_1+\chi_2}{2}\pm \sqrt{1+q^2+\alpha^2}\dfrac{\chi_1-\chi_2}{2}$, where $\chi_1,\chi_2:S\rightarrow\mathbb{C}$ are two different non-zero multiplicative functions and $q\in \mathbb{C}\backslash \lbrace \pm i\alpha\rbrace$ is a constant. Then
\[f=\left(\dfrac{\alpha-iq}{2} \right)\chi_1+\left(\dfrac{\alpha+iq}{2} \right)\chi_2,\]
and 
\[g=\left(\dfrac{-1-\sqrt{1+q^2+\alpha^2}}{2} \right)\chi_1+\left(\dfrac{-1+\sqrt{1+q^2+\alpha^2}}{2} \right)\chi_2,\]
or 
\[g=\left(\dfrac{-1+\sqrt{1+q^2+\alpha^2}}{2} \right)\chi_1+\left(\dfrac{-1-\sqrt{1+q^2+\alpha^2}}{2} \right)\chi_2.\]
Since $q\neq \pm i\alpha$, we see that $\alpha \neq \pm iq$. Then since $f$ is even and $g$ is odd we get  according to Lemma \ref{prob} (1) that
$$\left\{ \begin{array}{l}
   \dfrac{\alpha-iq}{2}=\dfrac{\alpha+iq}{2} \\ 
  \dfrac{-1+\sqrt{1+q^2+\alpha^2}}{2}+\dfrac{-1-\sqrt{1+q^2+\alpha^2}}{2}=0 \\ 
\end{array}.\right. $$
 This case does not occur.\\
(ii) $\alpha \neq 0$,  $if=i\alpha \chi_1$ and $-g=\chi_2$, where $\chi_1,\chi_2:S\rightarrow\mathbb{C}$ are two different non-zero multiplicative functions. Then $f=\alpha \chi_1$ and $g=-\chi_2$. In addition since $g$ is odd and $f$ is even, we deduce that $\chi_1^*=\chi_1$ and $\chi_2^*=-\chi_2$. This means that $\chi_2=0$. This case does not occur.\\
(iii) $if=i\alpha \chi+h$ and $-g=\chi\pm h$, where $\chi,h:S\rightarrow\mathbb{C}$ is a solution of \eqref{h}. So $f=\alpha \chi-ih$ and $g=-\chi\pm h$. Suppose that $g=-\chi-h$, we get that
\[f-ig=(\alpha+i)\chi,\]
then since $f$ is even and $g$ is odd we deduce that 
\[f+ig=(\alpha+i)\chi^*.\]
By adding the last two identities, we get that
\[f=\dfrac{\alpha+i}{2}\left[\chi+\chi^* \right], \]
and then 
\[g=\dfrac{\alpha+i}{2i}\left[\chi^*-\chi \right]. \]
The fact that $g\neq 0$ implies that $\chi\neq \chi^*$ and $\alpha \neq -i$. So by substituting the forms of $g$ and $f$ in \eqref{Hen} we get that for all $x,y\in S$
\[\chi(x)\left[ (i\alpha-1)\chi(y)-(i\alpha+1)\chi^*(y)\right]+\chi^*(x)(1-i\alpha)\left[\chi(y)-\chi^*(y) \right] =0 .\]
Then since $\chi\neq \chi^*$, we deduce that $1-i\alpha=0$. That is $\alpha=-i$. This contradict the fact that $\alpha \neq -i$, so this case does not occur. Now if $g=-\chi+h$, we get by the same way that $f=\dfrac{\alpha-i}{2}\left[\chi+\chi^* \right]$ and $g=\dfrac{\alpha-i}{2i}\left[\chi-\chi^* \right]$, where $\chi\neq \chi^*$ and $\alpha\neq i$, and then by substitution that $\alpha =i$. This case does not occur.\\
\underline{Subcase D :} $g^{\circ}\neq 0$ and $f^{\circ}\neq 0$.\\
\underline{Subcase D.1 :} $g^{\circ}= 0$ on $S^2$ and $f^{\circ}= 0$ on $S^2$. That is $g=g^e$ on $S^2$ and $f=f^e$ on $S^2$, and from \eqref{L2} we deduce that $f^e=cg^e$ on $S^2$ for some constant $c\in \mathbb{C}$. That is $f=cg$ on $S^2$, so the functional equation \eqref{Hen} can be written as follows
\[(1-c\alpha)g(x\sigma(y))=g(x)g(y)-f(x)f(y).\]
Since $g$ and $f$ are linearly independent, we see that $\beta:=1-c\alpha\neq 0$. Then 
\[g(x\sigma(y))=\dfrac{1}{\beta}g(x)g(y)-\dfrac{1}{\beta}f(x)f(y).\]
This means that the pair $\left(\dfrac{1}{\beta}g,\dfrac{i}{\beta}f \right) $ satisfies the functional equation \eqref{AE}, so taking into account that $g$ and $f$ are linearly independent we get from \cite[Theorem 4.2 ((4) and (5))]{Ase} that $\dfrac{1}{\beta}g$ and $\dfrac{i}{\beta}f$ are even. This is a contradiction since $g^{\circ}\neq 0$ and $f^{\circ}\neq 0$. This case does not occur.\\
\underline{Subcase D.2 :} $g^{\circ}= 0$ on $S^2$ and $f^{\circ}\neq 0$ on $S^2$. We get from \eqref{L1} that $f^e=0$, then from \eqref{L2} that $g^e=0$ on $S^2$. That is $g=0$ on $S^2$. The functional equation \eqref{Hen} becomes
\[\alpha f(x\sigma(y))=f(x)f(y)-g(x)g(y).\]
Since $g$ and $f$ are linearly independent, we see that $\alpha \neq 0$, so
\[f(x\sigma(y))=\dfrac{1}{\alpha}f(x)f(y)-\dfrac{1}{\alpha}g(x)g(y).\]
Since $g\neq 0$ ($g^{\circ}\neq 0$)  we get according to Lemma \ref{lemme}  that $f$ and $g$ are linearly dependent. This is a contradiction. This case does not occur.\\
\underline{Subcase D.3 :} $g^{\circ}\neq 0$ on $S^2$ and $f^{\circ}= 0$ on $S^2$. We deduce from \eqref{L1} that $g^e=0$, and then by using \eqref{L2} we get that $f^e=0$ on $S^2$. This implies that $f=0$ on $S^2$, and \eqref{Hen} can be written as follows
\[g(x\sigma(y))=g(x)g(y)-f(x)f(y).\]
Since $f\neq 0$($f^{\circ}\neq 0$) we get according to Lemma \ref{lemme}  that $f$ and $g$ are linearly dependent. This case does not occur.\\
\underline{Subcase D.4 :} $g^{\circ}\neq 0$ on $S^2$ and $f^{\circ}\neq 0$ on $S^2$. By using \eqref{L1} we get that 
\begin{equation}
g^e=\delta f^e,
\label{g1}
\end{equation}
for some constant $\delta \in \mathbb{C}$.\\
\underline{Subcase D.4.1 :} $f^e= 0$. So $g^e=0$, and then if we replace $y$ by $\sigma(y)$ the functional equation \eqref{Hen} becomes
\[g(xy)=-g(x)g(y)+f(x)f(y)+\alpha f(xy).\]
This means that the pair $\left(-g,f \right) $ satisfies \eqref{stet}, where $\alpha$ is replaced by $-\alpha$. According to \cite[Theorem 3.1]{ST2} and taking into account that $g$ and $f$ are linearly independent we have the following cases\\
(i) $f=-\alpha \dfrac{\chi_1+\chi_2}{2}+q\dfrac{\chi_1-\chi_2}{2}$ and $-g=\dfrac{\chi_1+\chi_2}{2}\pm \sqrt{1+q^2-\alpha^2}\dfrac{\chi_1-\chi_2}{2}$, where $\chi_1,\chi_2:S\rightarrow\mathbb{C}$ are two different non-zero multiplicative functions and $q\in \mathbb{C}\backslash \lbrace \pm \alpha\rbrace$ is a constant. Then $g=-\dfrac{\chi_1+\chi_2}{2}\pm \sqrt{1+q^2-\alpha^2}\dfrac{\chi_1-\chi_2}{2}$.\\
So $f-\alpha g=\dfrac{q\pm \alpha\sqrt{1+q^2-\alpha^2}}{2}\left(\chi_1-\chi_2 \right) $, and $\dfrac{q\pm \alpha\sqrt{1+q^2-\alpha^2}}{2}\neq 0$ since $g$ and $f$ are linearly independent. The fact that $f$ and $g$ are odd implies that
\[\dfrac{q\pm \alpha\sqrt{1+q^2-\alpha^2}}{2}(\chi_1-\chi_2)=\dfrac{q\pm \alpha\sqrt{1+q^2-\alpha^2}}{2}(\chi_2^*-\chi_1^*).\]
This implies that 
\[\chi_1-\chi_2=\chi_2^*-\chi_1^*.\]
Then, since $g$ is odd we get  that 
\[\chi_1+\chi_2=-(\chi_1^*+\chi_2^*).\]
So we deduce from the last two identities that $\chi_1=-\chi_1^*$ and $\chi_2=-\chi_2^*$. This implies that $\chi_1=\chi_2=0$. This case does not occur.\\
(ii) $\alpha \neq 0$,  $f=-\alpha \chi_1$ and $-g=\chi_2$, where $\chi_1,\chi_2:S\rightarrow\mathbb{C}$ are two different non-zero multiplicative functions. Then $g=-\chi_2$. In addition since $g$  and $f$ are odd, we deduce that $\chi_1^*=-\chi_1$ and $\chi_2^*=-\chi_2$. That is $\chi_1=\chi_2=0$. This case does not occur.\\
(iii) $f=-\alpha \chi+h$ and $-g=\chi\pm h$, where $\chi,h:S\rightarrow\mathbb{C}$ is a solution of \eqref{h}. So $g=-\chi\pm h$. Then $f-\alpha g=(1\pm \alpha )h$ and $1\pm \alpha \neq 0$ since $g$ and $f$ are linearly independent. Since $g$ and $f$ are odd, we see that $h=-h^*$, and then $\chi=-\chi^*$. This implies that  $\chi=0$. This case does not occur.\\
\underline{Subcase D.4.2 :} $f^e\neq 0$. We get from \eqref{L1} that $f^{\circ}=bg^{\circ}$ on $S^2$ for some constant $b\in \mathbb{C}$ and $\delta \neq 0$.\\
\underline{Subcase D.4.2.1 :} $f^e= 0$ on $S^2$. So $g^e=0$ on $S^2$. That is $f=bg$ on $S^2$ since $f^e=g^e=0$ on $S^2$. The functional equation \eqref{Hen} can be written as
\[(1-b\alpha)g(x\sigma(y))=g(x)g(y)-f(x)f(y).\]
The fact that $g$ and $f$ are linearly independent implies that $\gamma :=1-b\alpha \neq 0$, so we get
\[g(x\sigma(y))=\dfrac{1}{\gamma} g(x)g(y)-\dfrac{1}{\gamma} f(x)f(y).\]
This means that the pair $\left(\dfrac{1}{\gamma} g,\dfrac{i}{\gamma}f\right) $ satisfies \eqref{AE}, then since $f$ and $g$ are linearly independent we get from \cite[Theorem 4.2 ((4) and (5))]{Ase} that $\dfrac{1}{\gamma} g$ and $\dfrac{i}{\gamma} f$ are even. This contradicts the fact that $g^{\circ }\neq 0$ and $f^{\circ}\neq 0$. This case does not occur.\\
\underline{Subcase D.4.2.2 :} $f^e\neq 0$ on $S^2$. In this case we get from \eqref{L2} that $f^{\circ}=ag^{\circ}$ for some constant $a\in \mathbb{C}$, and $a\neq 0$ since $f^{\circ}\neq 0$. So
\[g^{\circ}=\dfrac{1}{a}f^{\circ}.\]
So with \eqref{g1}, equation \eqref{L1} becomes 
\[\dfrac{\delta}{a}f^e(x)f^{\circ}(yz)=f^e(x)f^{\circ}(yz),\]
for all $x,y,z\in S$. This implies that $\dfrac{\delta}{a}=1$, and then $\delta=a$, so
we get that 
\begin{equation}
g^{\circ}=\dfrac{1}{\delta}f^{\circ}.
\label{g2}
\end{equation}
By adding \eqref{g2} to \eqref{g1} we get that 
\begin{equation}
g=\delta f^e+\dfrac{1}{\delta}f^{\circ}.
\label{g3}
\end{equation}
Since $f$ and $g$ are linearly independent, we see from \eqref{g3} that $\delta \neq \pm 1$. So in view of \eqref{g3} the functional equation \eqref{Hen} becomes 
\begin{equation}
(\delta-\alpha)f^e(x\sigma(y))+\left(\dfrac{1}{\delta}-\alpha \right)f^{\circ}(x\sigma(y))=g(x)g(y)-f(x)f(y).
\label{s1} 
\end{equation}
If we apply \eqref{s1} first to the pair $(x,\sigma(y))$ and then to the pair $(\sigma(x),y)$  we get respectively the two following identities
\begin{equation}
(\delta-\alpha)f^e(xy)+\left(\dfrac{1}{\delta}-\alpha \right)f^{\circ}(xy)=g(x)g^*(y)-f(x)f^*(y).
\label{s2}
\end{equation}
\begin{equation}
(\delta-\alpha)f^e(xy)-\left(\dfrac{1}{\delta}-\alpha \right)f^{\circ}(xy)=g^*(x)g(y)-f^*(x)f(y).
\label{s3}
\end{equation}
By adding \eqref{s2} to \eqref{s3} and taking into account that $k^*=k^e-k^{\circ}$ and $k=k^e+k^{\circ}$ for any function $k:S\rightarrow\mathbb{C}$ we get
\begin{equation}
(\delta-\alpha)f^e(xy)=g^e(x)g^e(y)-g^{\circ}(x)g^{\circ}(y)-f^e(x)f^e(y)+f^{\circ}(x)f^{\circ}(y).
\label{s4}
\end{equation}
Now by using \eqref{g1} and \eqref{g2}, the identitiy \eqref{s4} becomes 
\begin{equation}
(\delta-\alpha)f^e(xy)=(\delta^2-1)f^e(x)f^e(y)+\left(1-\dfrac{1}{\delta^2} \right)f^{\circ}(x)f^{\circ}(y).
\label{s5} 
\end{equation}
Since $\delta\neq\pm 1$ and $f^e\neq 0$ on $S^2$, we see from \eqref{s5} that $\alpha \neq \delta$, and if we apply \eqref{s5} to the pair $(\sigma(y),x)$ we get 
\begin{equation}
(\delta-\alpha)f^e(\sigma(y)x)=(\delta^2-1)f^e(x)f^e(y)-\left(1-\dfrac{1}{\delta^2} \right)f^{\circ}(x)f^{\circ}(y).
\label{s6} 
\end{equation}
By adding \eqref{s6} to \eqref{s5} we get since $\alpha \neq \delta$
\begin{equation}
f^e(xy)+f^e(\sigma(y)x)=\dfrac{2(\delta^2-1)}{\delta-\alpha}f^e(x)f^e(y).
\label{s7}
\end{equation}
According to \cite[Theorem 2.1]{STd} we deduce from \eqref{s7} that 
\begin{equation}
f^e=\dfrac{\delta-\alpha}{2(\delta^2-1)} \left( \chi+\chi^*\right) ,
\label{s8}
\end{equation}
where $\chi:S\rightarrow\mathbb{C}$ is non-zero multiplicative since $f^e\neq 0$. Substituting \eqref{s8} in \eqref{s5} and taking into account that $f^{\circ}\neq 0$ we obtain 
\begin{equation}
f^{\circ}=\dfrac{\delta(\delta-\alpha)}{2(\delta^2-1)} \left( \chi-\chi^*\right),
\label{s9}
\end{equation}
where $\chi\neq \chi^*$, so we deduce from \eqref{s8} and \eqref{s9} that 
\begin{equation}
f=\dfrac{\delta-\alpha}{2(\delta^2-1)} \left( \chi+\chi^*\right)+\dfrac{\delta(\delta-\alpha)}{2(\delta^2-1)} \left( \chi-\chi^*\right),
\label{s10}
\end{equation}
and then by using \eqref{g1} and \eqref{g2} that 
\begin{equation}
g=\dfrac{\delta(\delta-\alpha)}{2(\delta^2-1)} \left( \chi+\chi^*\right)+\dfrac{\delta-\alpha}{2(\delta^2-1)} \left( \chi-\chi^*\right).
\label{s11}
\end{equation}
By substituting \eqref{s10} and \eqref{s11} in \eqref{Hen} we deduce after reduction that $\delta \alpha=1$. That is $\delta =\dfrac{1}{\alpha}$, and then we get that 
\[f=\dfrac{1+\alpha}{2}\chi-\dfrac{1-\alpha}{2}\chi^*\ \ \text{and}\ \ g=\dfrac{1+\alpha}{2}\chi+\dfrac{1-\alpha}{2}\chi^*,\]
where $\alpha \neq \pm 1$. This occurs in case (8).\par 
Finally, if $S$ is a topological semigroup, the continuity statements follows directly from \cite[Proposition 5.1]{ST2}. This completes the proof of Theorem \ref{thm1}.
\end{proof}
All of the following follows directly from Theorem \ref{thm1} and \cite[Theorem 3.1 (B)]{EB2}.
\begin{cor}
The solutions $g,f:S\rightarrow\mathbb{C}$ of the functional equation \eqref{Hen} are the following families:
\begin{enumerate}
\item[(1)] $\alpha = \pm 1$, $f$ is any non-zero function and $g=\alpha f$.
\item[(2)]$\alpha \neq 1$, $f=g\neq 0$ and $g=0$ on $S^2$.
\item[(3)] $\alpha \neq -1$, $f=-g\neq 0$ and $g=0$ on $S^2$.
\item[(4)] $f=(q+\alpha)\dfrac{\chi}{2}$ and $g=\left(1\pm \sqrt{1+q^2-\alpha ^2} \right)\dfrac{\chi}{2} $, where $q\in \mathbb{C}$ is a constant and $\chi :S\rightarrow \mathbb{C}$ a non-zero even multiplicative function.
\item[(5)] $f=\alpha \dfrac{\chi_1+\chi_2}{2}+q\dfrac{\chi_1-\chi_2}{2}$ and $g=\dfrac{\chi_1+\chi_2}{2}\pm \sqrt{1+q^2-\alpha^2}\dfrac{\chi_1-\chi_2}{2}$, where $\chi_1,\chi_2:S\rightarrow\mathbb{C}$ are two different non-zero even multiplicative functions and $q\in \mathbb{C}\backslash \lbrace \pm \alpha\rbrace$ is a constant.
\item[(6)] $\alpha \neq 0$,  $f=\alpha \chi_1$ and $g=\chi_2$, where $\chi_1,\chi_2:S\rightarrow\mathbb{C}$ are two different non-zero even multiplicative functions.
\item[(7)] $f=\alpha \chi+h$ and $g=\chi\pm h$ such that 
$$h=\left\{ \begin{matrix}
   \chi A & on & S\backslash {{I}_{\chi }}  \\
   0 & on & {{I}_{\chi }}\backslash {{P}_{\chi }}  \\
   \rho  & on & {{P}_{\chi }}  \\
\end{matrix}\ , \right.$$ 
where $\chi:S\rightarrow\mathbb{C}$ is an even non-zero multiplicative function, $A:S\backslash I_{\chi}\rightarrow\mathbb{C}$ an additive function such that $A\circ \sigma=A$, and $\rho:P_\chi\rightarrow\mathbb{C}$ a function such that $\rho\circ\sigma=\rho$. In addition we have the following conditions.\\
(I): If $x\in \lbrace up,pv,upv\rbrace$ for $p\in P_{\chi}$ and $u,v\in S\backslash I_{\chi}$, then we have respectively $\rho(x)=\rho(p)\chi(u)$, $\rho(x)=\rho(p)\chi(v)$, or $\rho(x)=\rho(p)\chi(uv)$.\\
(II): $h(xy)=h(yx)=0$ for all $x\in I_{\chi}\backslash P_{\chi}$ and $y\in S\backslash I_{\chi}$.
\item[(8)] $\alpha\neq \pm 1$, $f=\dfrac{1+\alpha}{2}\chi-\dfrac{1-\alpha}{2}\chi^*$ and $g=\dfrac{1+\alpha}{2}\chi+\dfrac{1-\alpha}{2}\chi^*$, where $\chi:S\rightarrow\mathbb{C}$ is a multiplicative function such that $\chi\neq \chi^*$.\\
Note that, off the exceptional case (1), $g$ and $f$ are abelian.\par 
Furthermore, if $S$ is a topological semigroup,  $f\in C(S)$ such that $f\neq \alpha \chi$ for any multiplicative function $\chi\in C(S)$, then $g\in C(S)$.
\end{enumerate}
\label{cor}
\end{cor}
\section{Examples}
In this section we give some examples of solutions of the functional equation \eqref{Hen}.
\begin{ex}
Let $S=(\mathbb{R},+)$ under the usual topology, and let $\sigma :\mathbb{R}\rightarrow\mathbb{R}$ be the involution defined by $\sigma(x)=-x$ for all $x\in \mathbb{R}$. The functional equation \eqref{Hen} can be written as
\begin{equation}
g(x-y)=g(x)g(y)-f(x)f(y)+\alpha f(x-y),
\label{e1}
\end{equation}
where $f,g:\mathbb{R}\rightarrow\mathbb{C}$. We determine the continuous solutions of \eqref{e1} with $\alpha\neq 0$. The case $\alpha=0$ is \cite[Example 4.18]{ST1} where the function $f$ is replaced by $if$. The continuous non-zero multiplicative functions on $S$ are the functions
\[\chi(x)=e^{i\lambda x},\quad x\in \mathbb{R},\]
where $\lambda \in \mathbb{C}$. In addition if $\chi$ is even then $\lambda=0$. The continuous additive functions on $S$ are of the form $a(x)=\delta x$ for all $x\in \mathbb{R}$ and some constant $\delta\in \mathbb{C}$. Such functions are even if and only if $\delta=0$. The solutions $f,g\in C(S)$ of \eqref{e1} are the following:
\begin{enumerate}
\item[(a)] $f=g=0$.
\item[(b)] $\alpha =\pm 1$, $f$ is any non-zero continuous function on $S$ and $g=\alpha f$.
\item[(c)] $f=\dfrac{q+\alpha}{2}$ and $g=\dfrac{1\pm \sqrt{1+q^2-\alpha^2}}{2}$, where $q\in \mathbb{C}$ is a constant.
\item[(d)] $f(x)=\alpha\cos(\lambda x)+i\sin(\lambda x) $ and $g(x)=\cos (\lambda x)+i\alpha\sin (\lambda x)$, where $\lambda \in \mathbb{C}\backslash \lbrace 0\rbrace$.
\end{enumerate}
\end{ex}
\begin{ex}
Let $S=H_3$ be the Heisenberg group defined by
\[H_3=\left\lbrace \left( \begin{matrix}
   1 & x & z  \\
   0 & 1 & y  \\
   0 & 0 & 1  \\
\end{matrix} \right)\mid \quad x,y,z\in \mathbb{R} \right\rbrace , \]
and let 
\[X=\left( \begin{matrix}
   1 & x & z  \\
   0 & 1 & y  \\
   0 & 0 & 1  \\
\end{matrix} \right),\]
for all $x,y,z\in \mathbb{R}$. We consider the following involution 
\[\sigma \left( X \right)=\left( \begin{matrix}
   1 & -x & z  \\
   0 & 1 & -y  \\
   0 & 0 & 1  \\
\end{matrix} \right).\]
According to \cite[Example 2.11, Example 3.14]{ST1}, the continuous  additive and the non-zero multiplicative functions on $S$ have respectively the forms
\[A \left( X \right)=\alpha x+\beta y ,\]
and 
\[\chi \left( X \right)=e^{ax+by},\]
where $\alpha ,\beta ,a ,b\in\mathbb{C}$. So $A\circ \sigma=A$ if and only $A=0$. On the other hand $\chi^*=\chi$ implies that $\chi =1$. So the continuous  solutions of equation \eqref{Hen} are the following four type: 
\begin{enumerate}
\item[(1)] $f=g=0$.
\item[(2)] $\alpha=\pm 1$, $f$ is any non-zero continuous function on $S$ and $g=\alpha f$.
\item[(3)] $f(X)=\dfrac{q+\alpha}{2}$ and $g(X)=\dfrac{1\pm \sqrt{1+q^2-\alpha^2}}{2}$, where $q\in \mathbb{C}$.
\item[(4)] $\alpha \neq \pm 1$, and
$$\left\{ \begin{matrix}
     f(X)=\dfrac{1+\alpha}{2}e^{ax+by}-\dfrac{1-\alpha}{2}e^{-ax-by} \\
   g(X)=\dfrac{1+\alpha}{2}e^{ax+by}+\dfrac{1-\alpha}{2}e^{-ax-by}  \\
\end{matrix}\quad , \right.$$
where $a,b\in \mathbb{C}$ are constants such that $(a,b)\neq (0,0)$.
\end{enumerate}
\end{ex}
In the following example we shall apply our theory to a semigroup $S$ such that $S^2\neq S$. 
\begin{ex}
Let $S=(\mathbb{N}\backslash \lbrace 1\rbrace,.)$, and let $\chi:S\rightarrow\mathbb{C}$ be the multiplicative function defined by 
$$\chi (x):=\left\{ \begin{matrix}
   1 & for & x\in \mathbb{N}\backslash \left( 2\mathbb{N}\cup \lbrace 1\rbrace\right)   \\
   0 & for & x\in 2\mathbb{N}  \\
\end{matrix}. \right.$$
Then $I_{\chi}=2\mathbb{N}$ and $P_{\chi}=2\mathbb{N}\backslash 4\mathbb{N}$. We let $\sigma:=id$ be the identity function. So $\chi$ is even with respect to $\sigma$. An additive function $A:\mathbb{N}\backslash \left( 2\mathbb{N}\cup \lbrace 1\rbrace\right)  \rightarrow\mathbb{C}$ is defined by
\[A(x):=\text{the number of times 5 occurs in the prime factorization of x}.\]
By Corollary \ref{cor} (7), the form of the function $h:\mathbb{N}\backslash \lbrace 1\rbrace\rightarrow\mathbb{C}$ is 
$$h(x):=\left\{ \begin{matrix}
   A(x) & for & x\in \mathbb{N}\backslash \left( 2\mathbb{N}\cup \left\{ 1 \right\} \right)  \\
   0 & for & x\in 4\mathbb{N}  \\
   \rho  & for & x\in 2\mathbb{N}\backslash 4\mathbb{N}  \\
\end{matrix} ,\right.$$
where $\rho:2\mathbb{N}\backslash 4\mathbb{N}\rightarrow \mathbb{C}$ is a function satisfying the condition (I). So
\[h(2(2n+1))=\rho(2)\chi(2n+1)=\rho (2)\quad\text{for all}\ n\in \mathbb{N}.\]
This implies that $\rho(x)=c$ for all $x\in 2\mathbb{N}\backslash 4\mathbb{N}$, where $c:=\rho(2)\in \mathbb{C}$. That is 
$$h(x)=\left\{ \begin{matrix}
   A(x) & for & x\in \mathbb{N}\backslash \left( 2\mathbb{N}\cup \left\{ 1 \right\} \right)  \\
   0 & for & x\in 4\mathbb{N}  \\
   c  & for & x\in 2\mathbb{N}\backslash 4\mathbb{N}  \\
\end{matrix} .\right.$$
So the condition (II) is satisfied. We get the solutions $f,g:S\rightarrow\mathbb{C}$ of equation \eqref{Hen} by plugging the appropriate forms above into the formulas of Corollary \ref{cor}.\par 
For the semigroup $S$ we have that $S^2=S\backslash \mathbb{P}$, where $\mathbb{P}$ denote the set of all prime numbers, and also $S$ is neither monoid nor generated by its squares, so we can see that we can not avoid the cases (2), (3) and (7) from Corollary \ref{cor}. 

\end{ex}
\subsection*{Declarations}

\textbf{Author contributions} This work is done  by the authors solely.\\
\\
\textbf{Funding} None.\\
\\
\textbf{Availability of data and materials} Not applicable.\\
\\
\textbf{Code Availability} Not Applicable.\\
\\
\textbf{Conflict of interest} None.

\end{document}